\documentclass[12pt]{article}
\usepackage{geometry}                
\usepackage{graphicx,color}
\usepackage{amssymb,amsmath,amsthm,mathrsfs}
\usepackage[all,cmtip]{xy}
\usepackage{epstopdf, comment}
\usepackage{esint}

\usepackage[pdftex,bookmarks,pdfnewwindow,plainpages=false,unicode,pdfencoding=auto]{hyperref}

 \hypersetup{
pdfauthor={Oleg Ivrii},
pdftitle={Sparse Beltrami coefficients, integral means of conformal mappings and the Feynman-Kac formula},
pdfsubject={Geometric function theory},
bookmarksdepth={4}
}

\DeclareMathOperator{\Mdim}{M.dim }
\DeclareMathOperator{\supp}{supp }

\newtheorem{lemma}{Lemma}[section]
\newtheorem{theorem}{Theorem}[section]
\newtheorem{corollary}{Corollary}[section]

\theoremstyle{remark}
\newtheorem*{remark}{Remark}

\numberwithin{equation}{section}

\DeclareMathOperator{\im}{Im}

\DeclareMathOperator{\diam}{diam}
\DeclareMathOperator{\aut}{Aut}

\DeclareMathOperator{\hyp}{hyp}

\newcommand{\dzy}{\,\frac{|dz|^2}{y}}

\newcommand{\Hbar}{\overline{\mathbb H}}

\DeclareMathOperator{\middy}{mid}

\title{Sparse Beltrami coefficients, integral means of conformal mappings and the Feynman-Kac formula}
\author{Oleg Ivrii}
\date{January 3, 2017}  
 
\begin{document}

\maketitle

\begin{abstract}
In this note, we  give an estimate for the  dimension of the image of the unit circle under a quasiconformal mapping whose dilatation has small support. We also prove an analogous estimate 
for the rate of growth of a solution of  a second-order parabolic equation given by the Feynman-Kac formula (with a sparsely supported potential)
and introduce a dictionary between the two settings. 
\end{abstract}

\section{Introduction}
 
  For a Beltrami coefficient $\mu$ with  $\|\mu\|_\infty < 1$, let $\tilde w^{\mu}$ denote the normalized solution of the Beltrami equation $\overline{\partial} w = \mu \, \partial w$ which fixes the points $0,1,\infty$. In the classical question on dimensions of quasicircles, one is interested in maximizing the Minkowski dimension of 
 $\tilde w^{\mu}(\mathbb{S}^1)$ over all Beltrami coefficients on the plane with $\|\mu\|_\infty \le k$ for a fixed $0 \le k < 1$. This question has been extensively studied by many authors. One notable result in this area is due to S.~Smirnov \cite{smirnov} who gave the upper bound \begin{equation}
D(k) \,=\, \sup_{\|\mu\|_\infty \le k} \Mdim \tilde w^{\mu}(\mathbb{S}^1) \,\le\, 1 + k^2, \qquad \text{for all }\ k \in [0,1),
\end{equation}
while recently, it was observed that $D(k) = 1 + k^2 \Sigma^2 + \mathcal O(k^{8/3 - \varepsilon})$ for some constant $0.87913 < \Sigma^2 < 1$, see the works \cite{AIPP, hedenmalm, qcdim}. 
%

In the present paper, we are interested in the case when the support of $\mu$ is contained in a garden 
\begin{equation}
\label{eq:the-garden}
\mathcal G = \bigcup_{j=1}^\infty B_j, \qquad d_{\mathbb{D}}(B_i, B_j) > R, \quad i \ne j,
\end{equation} made up of countably many horoballs $B_j \subset \mathbb{D}$, any two of which are at least a distance $R$ apart in the hyperbolic metric. Equivalently, we require   the horoballs $B^*_j = \{ z \in \mathbb{D} : d_{\mathbb{D}}(z, B_j) \le R/2 \}$ to be disjoint.

Let $\mu^+ = \overline{(1/\overline{z})^*\mu}$ denote the reflection of $\mu$ in the unit circle. Since the support of
 $\mu^+$ is contained in the exterior unit disk, $\tilde w^{\mu^+}: \mathbb{D} \to \mathbb{C}$ is conformal. 
Our first main theorem states:

\begin{theorem}
\label{sparse-thm}
Suppose $\mu$ is a Beltrami coefficient on the unit disk with $\| \mu \|_\infty \le 1$. If $\mu$ has sparse support, then 
\begin{equation}
\label{eq:dim-statement}
\Mdim \tilde w^{k \mu^+}(\mathbb{S}^1) \le 1 + Ce^{-R/2} k^2, \quad k < \min \Bigl (0.49, \frac{c}{2R} \Bigr ).
\end{equation}
\end{theorem}

\subsection{Integral means spectra}

To prove Theorem \ref{sparse-thm}, we will analyze integral means of conformal mappings. For a conformal mapping $f: \mathbb{D} \to \mathbb{C}$, its {\em integral means spectrum} is given by
\begin{equation}
\label{eq:def-beta}
\beta_f(p) = \limsup_{r\to 1^-} \frac{\log \int_{|z|=r} |f'(z)^p| \, d\theta }{ \log \frac{1}{1-r}}, \qquad p \in \mathbb{C}.
\end{equation}
The connection between integral means and the Minkowski dimension of the boundary of the image domain comes from the relation
\begin{equation}
\label{eq:pomm}
\beta_{f}(p) = p-1 \Longleftrightarrow  p = \Mdim f(\mathbb{S}^1),
\end{equation}
valid when $f(\mathbb{S}^1)$ is a quasicircle, see for instance \cite[Corollary 10.18]{Pomm}. 
In view of the above identity, to prove Theorem \ref{sparse-thm}, 
it suffices to show:

\begin{theorem}
\label{sparse-thm2}
\begin{equation}
\beta_{\tilde w^{k\mu^+}}(p) \le Ce^{-R/2}k^2|p|^2/4, \qquad k < 0.49, \quad k|p| < c/R.
\end{equation}
\end{theorem}


The advantage of estimating integral means comes from the fact that they allow us to view Theorem \ref{sparse-thm} as a {\em growth problem}\/. To make this feature more visible, consider the {\em Brownian spectrum of a conformal mapping}\/:
\begin{align}
\tilde \beta_f(p) &  := \limsup_{t \to \infty} \frac{1}{t} \log \mathbb{E}_0 |f'(B_t)|^p, \\
& = \limsup_{t \to \infty} \frac{1}{t} \log \int_{\mathbb{D}} |f'(z)|^p \cdot p_t(0, x) dA_{\hyp}(x).
\end{align}
In the equations above, $B_t$ is hyperbolic Brownian motion, that is,  Brownian motion in $\mathbb{D}$ equipped with the hyperbolic
metric $\rho = \frac{2 |dz|}{1-|z|^2}$, while 
the subscript ``0'' in $\mathbb{E}_0$ indicates that Brownian motion is to be started at the origin. Finally, $p_t(0,x)$ is the hyperbolic heat kernel which measures the probability density that a Brownian particle travels from $0$ to $x$ in time $t$.

In Section \ref{sec:BIMS}, we will show that $\tilde \beta_f(p) \ge \beta_f(p)$ for any $p \in \mathbb{C}$; this follows from the fact that the expected displacement of hyperbolic Brownian motion is linear in time.
Therefore, to prove Theorem \ref{sparse-thm2}, we may estimate the Brownian spectrum instead.

\begin{remark}
The use of hyperbolic Brownian motion is inspired by the work  \cite{lyons} of T.~Lyons who give an alternative perspective on Makarov's {\em law of iterated logarithm} for Bloch functions. As noted by I.~Kayumov in \cite{kayumov}, the study of the behaviour of the integral means spectrum at the origin is slightly more general, so our considerations may be viewed as an extension of Lyons' ideas.
\end{remark}

\subsection{Growth of solutions of PDEs}

Our proof of Theorem  \ref{sparse-thm2}  is inspired by an analogous statement from parabolic PDEs whose solution is given by the Feynman-Kac formula. Let $\Delta_{\hyp} = \rho(x)^{-2} \Delta$ denote the hyperbolic Laplacian and
$A_{\hyp} = \rho(x)^2 dA$ be the hyperbolic area element. For a  positive and bounded potential $V$, we consider the second order parabolic differential equation
\begin{equation}
\label{eq:heat-equation4}
\frac{\partial u}{\partial t} = \frac{1}{2} \cdot \Delta_{\hyp} u + V(x) u(x,t), \qquad (x,t) \in \mathbb{D} \times (0,\infty),
\end{equation}
where the initial condition $u_0(x) = u(x, 0)$ is a positive compactly supported function. If the potential $V \equiv 0$, then (\ref{eq:heat-equation4}) reduces to the heat equation.
There are many possible ways to measure the growth of solutions of PDEs, but for our purposes, the {\em Lyapunov exponent}
\begin{equation}
\label{eq:lyapunov}
\beta_V := \limsup_{t \to \infty} \frac{1}{t} \log \int_{\mathbb{D}} u_t(x) dA_{\hyp}(x)
\end{equation}
is the most natural.
 As is well known, e.g.~see \cite[Section 4.3]{durrett}, the solution to (\ref{eq:heat-equation}) is given by {\em Feynman-Kac formula}
\begin{equation}
\label{eq:FK}
u_t(x) = \mathbb{E}_x \biggl \{ u_0(B_t) \exp \int_0^t V(B_s) ds \biggr \}.
\end{equation}
Since the initial condition $u_0$ was positive, $u_t$ will remain positive for all time, and the total mass of the solution will be increasing in $t$.

Our second main theorem is an analogue of Theorem \ref{sparse-thm2} for sparse potentials:
\begin{theorem}
\label{sparse-analogue}
For a potential $V = \chi_\mathcal G = \bigcup B_j$ with sparse support,
$$
\beta_V(p) \, := \, \beta_{pV} \, \le \, Ce^{-R/2} p^2,
$$
for $p < p_0(R)$ sufficiently small.
\end{theorem}

\begin{remark}
The choice of notation is justified by the observation that the Lyapunov spectra $\beta_{pV}$ posses many of the same properties as integral means spectrum of conformal mappings, for instance, they are increasing and convex functions in $p$.
\end{remark}

\subsection{Alternative ideas and remarks}

The main difficulty in Theorem \ref{sparse-thm} is to make use of the separation condition between horoballs.
Perhaps the most obvious attempt is to take a Beltrami coefficient $\mu$ supported on the garden
and cook up a Beltrami coefficient $\nu \in M(\mathbb{D})$ with  $\| \nu \|_\infty < \| \mu \|_\infty$ and
$\tilde w^\mu = \tilde w^\nu$ on $\mathbb{S}^1$ and  then apply the general bounds on dimensions of quasicircles mentioned above.
However, it is easily seen that such an approach is impossible: there are sparse $k$-quasicircles which are genuine $k$-quasicircles.   In fact, this is true even if $\mathcal G \subset \mathbb{D}$ is composed of a single horoball -- we only need  $\mathcal G$ to contain round balls  of arbitrarily large hyperbolic diameter.
This easily follows from the fact that the Teichm\"uller norm can be described by the dual pairing $$\| \mu \|_T = \inf_{\nu \sim \mu} \| \mu \|_\infty = \sup_{\|q\| = 1} \int_{\mathbb{D}} \mu \cdot q,$$ where the supremum is taken over all
integrable quadratic differentials $q$ on the unit disk with $\|q\| = \int_{\mathbb{D}} |q| = 1$.

The reader may try to improve on the arguments of Smirnov \cite{smirnov} who used complex interpolation techniques to give an elegant proof of the bound $D(k) \le 1 + k^2$ suggested by Astala \cite{A}. However, in this set of ideas, it seems unlikely that one can make use of the sparsity assumption on the support. Another natural approach is to extend
the arguments of C.~Bishop \cite[Lemma 6.4]{bishop-bigdef} which involve a corona-type construction. While these  ideas do utilize the sparsity of the support, it is not really clear how to exploit the martingale nature of Bloch functions in this context, which is necessary to obtain quadratic growth.

In \cite[Section 9]{qcdim}, an analogue of Theorem \ref{sparse-thm} was proved using the techniques of Becker and Pommerenke for estimating integral means of univalent functions for gardens
 $\mathcal G$  that are unions of thickened geodesics (unit neighbourhoods of hyperbolic geodesics) with the same separation condition. The case of horoballs introduces non-uniformity and therefore requires new ideas. A priori, it was not clear to the author how to extend these arguments either,
although it can be done -- see Section \ref{sec:bp-argument}. This approach can be  generalized to gardens composed of objects other than horoballs.

Another perspective was offered by N.~Michalache via the notion of {\em mean wiggly sets} from \cite{GJM}. Here, one seeks to decompose $\mathbb{S}^1 = B \sqcup G$ so that $\tilde w^{k\mu^+}(G)$ satisfies a mean wiggliness condition while $B$ has dimension less than $1-\varepsilon$. H\"older properties of quasiconformal mappings guarantee that for small $k > 0$, $\Mdim \tilde w^{k\mu^+}(\mathbb{S}^1) = \Mdim \tilde w^{k\mu^+}(G)$, from which point the mean wiggly machinery can be applied. 
A possible definition of $G$ could be $G = \{e^{i\theta} : \beta(\theta) < Ce^{-R/2} \}$ where
\begin{equation}
\label{eq:mw}
\beta(\theta) = \limsup_{r \to 1} \frac{ \ell_{\hyp} \bigl ( [0, r e^{i\theta}] \cap \mathcal G \bigr )}{\ell_{\hyp} \bigl ( [0, r e^{i\theta}] \bigr )},
\end{equation}
with $\ell_{\hyp}$ being the hyperbolic length. 

The author's original proof of Theorem \ref{sparse-thm}, which will be presented in Section \ref{sec:BIMS}, was motivated by the observation that hyperbolic 
Brownian motion started at a point $z_0 \notin \mathcal G$ spends little time in the garden:
for almost every Brownian path $B_t$,
\begin{equation}
\label{eq:smw}
\limsup_{t \to \infty} \frac{\int_0^t \chi_{\mathcal G}(B_s) ds}{t} \le Ce^{-R/2}.
\end{equation}
This may be viewed as a  stochastic analogue of mean wiggliness. Since we will not actually  use   (\ref{eq:smw}) in this paper,
we will not give a proof.



\subsection{Dynamical considerations}

Suppose
$\Gamma$ is a cofinite area Fuchsian group
with at least one cusp. One may construct a Beltrami coefficient $\mu \in M(\mathbb{D})^\Gamma$ satisfying the hypotheses of the theorem by 
lifting a
Beltrami coefficient on $\mathbb{D}/\Gamma$ supported on a collar neighbourhood of one of the cusps.
For these special dynamical coefficients, one can give an alternative proof of Theorem \ref{sparse-thm} based on
 McMullen's identity \cite{mcmullen, wpgeo} which says that
\begin{equation}
\frac{d^2}{dt^2}\biggl |_{t=0} \Mdim \tilde w^{t\mu}(\mathbb{S}^1)  = \frac{4}{3} \cdot \lim_{r \to 1^-} \frac{1}{2\pi} \int_{|z|=r} \biggl | \frac{v'''_{\mu^+}}{\rho^2}(re^{i\theta}) \biggr |^2 \, d\theta,
\end{equation}
where
$$
v'''_{\mu^+}(z) = -\frac{6}{\pi} \int_{|z|>1} \frac{\mu(\zeta)}{(\zeta-z)^4} |d\zeta|^2.
$$
As explained in \cite{mcmullen}, since $v'''_{\mu^+}$ is naturally a quadratic differential, to measure its size, one should divide by the square of the Poincar\'e metric. 
According to \cite[Section 2]{wpgeo}, one has
\begin{align}
\frac{d^2}{dt^2} \biggl |_{t=0} \Mdim \tilde w^{t\mu}(\mathbb{S}^1) & \le C \cdot \lim_{r \to 1^-}  |\mathcal G \cap S_r|, \\
&
\label{eq:mcm-cesaro}
 \le C \cdot \limsup_{r \to 1^-} \frac{1}{|\log (1-r)|} \int_0^r |\mathcal G \cap S_s| \, \frac{ds}{1-s},
\end{align}
where $S_r = \{ z : |z| = r\}$.
It is not difficult to see that for a garden $\mathcal G$ which is the union of horoballs a hyperbolic distance $R$ apart, the quantity
 (\ref{eq:mcm-cesaro}) is $\le Ce^{-R/2}$, which is  essentially the estimate we want.
(To be honest, it is not clear how to use thermodynamic formalism to obtain a uniform estimate for $\Mdim \tilde w^{t\mu}(\mathbb{S}^1)$
 independent of $\mu$ and the underlying dynamical system.)

In any case, McMullen's identity fails for general Beltrami coefficients $\mu$, and the bound (\ref{eq:mcm-cesaro})
 is not enough since one can
concentrate the Beltrami coefficient to expand a small arc of the unit circle and still obtain maximal dimension distortion.

\subsection*{Acknowledgements}

The author wishes to thank Kari Astala, Ilia Binder, Istv\'an Prause and Eero Saksman for interesting conversations. 
The research was funded by the Academy of Finland, project nos. 271983 and 273458.

\section{Background in probability}

\label{sec:probability}

We now gather some basic facts and definitions that will be used throughout this paper. 
Let $\Delta_{\hyp} = \rho(x)^{-2} \Delta$ denote the hyperbolic Laplacian and
$A_{\hyp} = \rho(x)^2 dA$ be the hyperbolic area element.
Consider the {\em hyperbolic heat equation}
\begin{equation}
\label{eq:old-heat-equation}
\frac{\partial u}{\partial t} = \frac{1}{2} \cdot \Delta_{\hyp} u,  \qquad (x,t) \in \mathbb{D} \times (0,\infty).
\end{equation} 
The  {\em hyperbolic heat kernel}  $p_t(x,y)$ is characterized by the property that if $u_0(x)$ is a bounded continuous function on the disk then
\begin{equation}
\label{eq:reproducing-property}
u_t(x) = \int_{\mathbb{D}} p_t(x,y) u_0(y) dA_{\hyp}(y)
\end{equation}
is the unique bounded solution of (\ref{eq:old-heat-equation}) with $\lim_{t \to 0} u_t(x) = u_0(x)$, see for instance, \cite[Theorem 4.1.3]{durrett}. Since this reproducing  property determines $p_t(x,y)$ uniquely, the heat kernel is conformally invariant and symmetric, that is,
  $p_t(\phi(x), \phi(y)) = p_t(x, y) = p_t(y,x)$ for any $\phi \in \aut \mathbb{D}$.
  
As noted in the introduction, $p_t(x,y)$ measures the probability density for a Brownian particle to go from $x$ to $y$ in time $t$. More precisely, for any measurable set $E \subset \mathbb{D}$,
\begin{equation}
\label{eq:hbm}
\mathbb{P}_x(B_t \in E) = \int_E p_t(x,y) dA_{\hyp}(y),
\end{equation}
where the subscript ``$x$'' in $\mathbb{P}_x$ denotes the fact that Brownian motion is started at $x$.
In fact, (\ref{eq:hbm}) may be taken as the definition of hyperbolic Brownian motion.

 For $t > 0$, we define the {\em partial Green's function} as $g_t(x,y) := \int_0^t p_s(x,y) ds$.
Taking $t = \infty$ gives the usual {\em Green's function} $g_\infty(x,y)$. The Green's function measures the occupation density of Brownian motion, that is,  the integral $\int_E g_\infty(x,y) dA_{\hyp}(y)$ computes the expected amount of time Brownian motion spends in $E$.

When $y=0$, we will shorten notation to $p_t(x) := p_t(0,x)$ and $g_t(x) := g_t(0,x)$. 
We will make use of the explicit formula
\begin{equation}
\label{eq:ginf}
g_\infty(x) = \frac{1}{\pi} \log \frac{1}{|x|}.
\end{equation}
According to \cite[Section 4.8]{durrett}, $g_\infty(x)dA(x) = \frac{1}{\pi} \log \frac{1}{|x|} dA(x)$ is the occupation measure for the Euclidean Brownian motion in $\mathbb{C}$ started at the origin and simulated until it hits $\partial \mathbb{D}$. This coincidence is explained by the fact that hyperbolic Brownian motion is a time change of the Euclidean Brownian motion by $\rho(x)^{2}$, that is if one wants to simulate hyperbolic Brownian motion up to time $t$, one can instead simulate Euclidean Brownian motion $W_2$ up to time $\tau$ defined by $t = \int_0^\tau \rho^{2} (W_2(s)) ds$.

The existence of the Green's function implies that hyperbolic Brownian motion is transient: any path tends to the unit circle.
In fact, the hyperbolic distance from the starting point to $B_t$ grows linearly with time, i.e.~for any $\varepsilon > 0$, 
\begin{equation}
\label{eq:linear-displacement}
\mathbb{P}_0 \Bigl ( (1-\varepsilon) t \,<\, d_\mathbb{D}(0,B_t) \,<\, (1 + \varepsilon)t \Bigr ) \,  \to \, 0, \qquad \text{as } t \to \infty,
\end{equation}
see for instance \cite[Theorem 3.1]{DM}.

\subsection{Brownian motion escapes from horoballs}

We now compute the expected time that Brownian motion spends in a horoball:
\begin{lemma}
\label{bm-hates-horoballs}
Let $B$ be a horoball in the unit disk and consider hyperbolic Brownian motion started at a point $z_0 \in \mathbb{D}$. Let
$\ell(B) =  \int_0^\infty   \chi_B(B_s) ds$ denote the amount of time Brownian motion spends in $B$. Then,

{\em (i)} If $z_0 \in \partial B$, $\mathbb{E}_{z_0}\bigl ( \ell(B) \bigr ) \asymp 1$.

{\em (ii)} If $z_0 \, \in \, \partial B^* = \{ z \in \mathbb{D} : d_{\mathbb{D}}(z, B) \le R/2 \},$ then  $\mathbb{E}_{z_0}\bigl ( \ell(B) \bigr ) \asymp e^{-R/2}$.
\end{lemma}

\begin{proof}
(i) We compute the expected amount of time that Brownian motion spends in $B$ by integrating the Green's function $g_\infty$:
$$
\mathbb{E}_{z_0} (\ell(B)) = \int_B g_\infty(z_0,z) dA_{\hyp}(z).
$$
By the conformal invariance of Brownian motion, this integral is independent of the choice of horoball $B$ and the point $z_0 \in \partial B$.
For a horoball $B$ passing through $z_0 = 0$, this is
$$
\mathbb{E}_{0}(\ell(B))  = \frac{1}{\pi} \int_B \log \frac{1}{|z|} dA_{\hyp}(z) < \infty.
$$
(ii) Similarly, conformal invariance allows us to consider the case when the initial point $z_0 = 0$ and $B \subset \mathbb{D}$ is a horoball that rests on 1. The assumption on the hyperbolic distance from $z_0$ to $B$ implies that the Euclidean diameter of $B$ is $\asymp e^{-R/2}$. Integrating, we obtain
$$
\mathbb{E}_{0}(\ell(B))  = \frac{1}{\pi} \int_B \log \frac{1}{|z|} dA_{\hyp}(z) \asymp e^{-R/2}
$$
as desired.
\end{proof}

For a horoball $B$, define the {\em excursion time} of Brownian motion as the length of time between the first entry and last exit. We now show that if Brownian motion enters a horoball, the probability of a long excursion is exponentially small:

\begin{lemma}
\label{bm-hates-horoballs2}
Let $B$ be a horoball in the unit disk and consider hyperbolic Brownian motion started at a point $z_0 \in \partial B$. Let $\tau_B = \sup_{t \ge 0}(B_t \in B)$. Then,
 $\mathbb{P}_{z_0}\bigl ( \tau_B > t \bigr ) < Ce^{-\gamma t}$ for some $\gamma > 0.$
\end{lemma}

\begin{proof}
We switch over to the upper half-plane. In view of the conformal invariance, we may take the initial point $z_0 = i$ and the horoball $$B = \{z \in \mathbb{H} : \im z > 1\}.$$
In the upper half-plane, the law of the imaginary part of hyperbolic Brownian motion $Z_t$ started at $i \in \mathbb{H}$ has a simple form: namely, $\log(\im Z_t) = 2W_t - t$ where $W_t$ is the usual Euclidean Brownian motion on the real line, e.g.~see \cite{lyons}. 
In this context, $2W_t -t \to -\infty$  almost surely and $\tau_B$  is the last time at which $2 W_t - t = 0$. As is well known to probabilists, the density of $\tau_B$ can be expressed in terms of the {\em gamma distribution}
$$\Gamma \biggl (\frac{1}{2}, \frac{1}{8} \biggr) = \frac{1}{\sqrt{8 \pi s}} \cdot e^{-s/8} ds, \qquad s \in (0, \infty),$$
so any exponent $\gamma < 1/8$ works.
\end{proof}

\subsection{Two lemmas on monotonicity}

For a horoball $B$ in the disk  which does not contain the origin, we denote its top point (the one closest to the origin) by  $z_B$ and its Euclidean center by  $z_B^{\middy}$.
The following lemma will be useful in the sequel:

 \begin{lemma}
 \label{t-replacement}
We have:
 
{\em (i)} For a fixed $t > 0$, the quotient $g_t(r)/g_\infty(r)$ is decreasing in $r \in [0,1)$.

{\em (ii)} For a horoball $B$ in the unit disk contained in $\{z \in \mathbb{D}: 1/2 < |z| < 1\}$,
$$ g_t(z_{B}^{\middy}) \, \lesssim \,  \int_{B}  g_t(x) dA_{\hyp}(x) \, \lesssim \, g_t(z_{B}).$$
\end{lemma}

\begin{proof}
(i)
 Let us show that $g_t(r_1)/g_\infty(r_1) < g_t(r_2)/g_\infty(r_2)$ with $r_1 < r_2$. Consider two very  thin disjoint annuli $A_1 = A(r_1, r_1')$ and $A_2 = A(r_2, r_2')$.
 From the probabilistic interpretation of the Green's function, it is clear that the function
$$
G_1(x) = \int_{A_1} g_\infty(x, y) dA_{\hyp}(y)
$$
only depends on $|x|$, is constant in $B(0, r_1)$ and is decreasing for $|x| \ge r_1$.
We denote the value of $G_1$ on $B(0, r_1)$ by $E_1$. We define $G_2$ and $E_2$ similarly using the annulus $A_2$ in place of $A_1$.

We claim that for any $x \in \mathbb{D}$,
\begin{equation}
\label{eq:t-replacement}
\frac{1}{E_2} \cdot  \int_{A_2} g_{\infty}(x, y) dA_{\hyp}(y) \ge \frac{1}{E_1} \cdot  \int_{A_1} g_{\infty}(x, y) dA_{\hyp}(y).
\end{equation} 
  There are three possibilities: either 
  $x \in B(0,r_1)$, $A(r_1, r_2)$ or $A(r_2, 1)$. 
  We examine the three cases separately.  
    In the first case, 
    $$
    \int_{A_1} g_{\infty}(x, y) dA_{\hyp}(y) = E_1, \qquad  \int_{A_2} g_{\infty}(x, y) dA_{\hyp}(y) = E_2,
    $$
     so (\ref{eq:t-replacement}) is an equality. In the second case,
         $$
    \int_{A_1} g_{\infty}(x, y) dA_{\hyp}(y) < E_1, \qquad  \int_{A_2} g_{\infty}(x, y) dA_{\hyp}(y) = E_2,
    $$
Finally, in the third case when $x \in A(r_2, 1)$, if a path of Brownian motion started at $x$ hits $A_1$, it must first cross the circle $S_{r_2}$, so the third case can only be worse than the second case. This proves (\ref{eq:t-replacement}).

Let $g_{t, \infty} = g_\infty - g_t$.  
From (\ref{eq:t-replacement}) and the Markov property of Brownian motion, it is clear that
\begin{equation}
\label{eq:t-replacement}
\frac{1}{E_2} \cdot  \int_{A_2} g_{t,\infty}(0, y) dA_{\hyp}(y) \ge \frac{1}{E_1} \cdot  \int_{A_1} g_{t,\infty}(0, y) dA_{\hyp}(y).
\end{equation} 
Statement (i) follows by tending the thickness of the annuli $A_1$ and $A_2$ to zero.

(ii) For $t=\infty$, this is a simple computation based on the explicit expression for $g_\infty$. 
For $t < \infty$, we use (i). The upper bound is immediate, while for the lower bound, we only need to estimate the integral over the top half of $B$.
\end{proof}

We will also need:

\begin{lemma}
 \label{s-replacement}
Suppose $B_1 \subset B_2$ are two horoballs which rest on the same point of the unit circle and $z_0 \in \partial B_2$. Then, for any $t > 0$,
$$
\frac{\int_{B_1} g_t(z_0, x) dA_{\hyp}(x)}{\int_{B_1} g_\infty(z_0, x) dA_{\hyp}(x)}
< \frac{\int_{B_2} g_t(z_0, x) dA_{\hyp}(x)}{\int_{B_2} g_\infty(z_0, x) dA_{\hyp}(x)}.
$$
\end{lemma}

The proof is similar to that of the previous lemma except one considers thin crescents instead of annuli.
We leave the details to the reader.

\section{Feynman-Kac formula}
\label{sec:FK}

In this section, we discuss an analogous problem in the context of parabolic PDEs. Consider a potential $V: \mathbb{D} \to \mathbb{R}$, which we assume
to be positive and bounded. 
 We are interested in studying the growth of solutions of the second order parabolic differential equation
\begin{equation}
\label{eq:heat-equation}
\frac{\partial u}{\partial t} = \frac{1}{2} \cdot \Delta_{\hyp} u + V(x) u(x,t), \qquad (x,t) \in \mathbb{D} \times (0,\infty),
\end{equation}
where the initial condition $u_0(x) = u(x, 0)$ is a positive compactly supported function.

\subsection{Basic properties}

Since $p_t(x, y) \to 0$ as $d_{\mathbb{D}}(x,y) \to \infty$, the Feynman-Kac formula (\ref{eq:FK}) guarantees that $u_t$ vanishes on the unit circle, i.e.~$u_t(x) \to 0$ as $|x| \to 1$. 
Applying Green's formula, we find
$$
\int_{\mathbb{D}} \Delta_{\hyp} u_t(x) dA_{\hyp}(x) \, = \, \int_{\mathbb{D}} \Delta u_t(x) dA(x)  \, = \, 0.
$$
In light of the above identity, if we integrate (\ref{eq:heat-equation}) over the unit disk, we obtain the crucial formula
\begin{equation}
\label{eq:heat-derivative}
\frac{d}{dt} \int_{\mathbb{D}} u_t(x) dA_{\hyp}(x) = \int_{\mathbb{D}} V(x) u_t(x) dA_{\hyp}(x),
\end{equation}
which says that locally near a point $x \in \mathbb{D}$, the mass of $u_t$ grows at rate $V$.
We now give an alternative formula for the Lyapunov exponent (\ref{eq:lyapunov}):

\begin{theorem}
\label{beta-growth}
The rate of growth of the solution is given by
\begin{equation}
\label{eq:beta-growth}
\beta_V = \limsup_{t \to \infty} \frac{1}{t} \cdot \mathbb{E}_0 \biggl \{ \exp \int_0^t V(B_s)ds \biggr \},
\end{equation}
irrespective of the initial condition.

\end{theorem}

\begin{proof}
From the Feynman-Kac formula, it is easy to see that
\begin{equation}
\label{eq:total-mass-u}
\int_{\mathbb{D}} u_t(x) dA_{\hyp}(x) = \int_{\supp u_0} u_0(x) \cdot \mathbb{E}_x \biggl \{ \exp \int_0^t V(B_s) ds \biggr\} 
  dA_{\hyp}(x).
\end{equation}
Indeed, one only needs to use the symmetry of the Brownian transition function, $p_t(x,y) = p_t(y, x)$.
From  \cite[Theorem 3.1]{DM}, it is clear that $$p_1(x_1, y)  < \phi \bigl (d_{\mathbb{D}}(x_1,x_2) \bigr ) \cdot p_2(x_2, y)$$ for some increasing function $\phi: (0, \infty) \to (0, \infty)$.
Combining this observation with the Markov property of Brownian motion reveals that
$$
v_t(x) := 
\mathbb{E}_x \biggl \{ \exp \int_0^t V(B_s) ds \biggr\}
$$
satisfies
 $$
\frac{v_{t}(x_1)}{v_{t+2}(x_2)} < \phi \bigl (d_{\mathbb{D}}(x_1,x_2) \bigr ) \cdot e^{\|V\|_\infty}, \qquad t \ge 1.
$$
Reversing the roles of $x_1$ and $x_2$ gives an inequality in the other direction. 

The above considerations show that the growth rates of all  functions $v_t(x)$, $x \in \mathbb{D}$ are the same.
However, according to (\ref{eq:total-mass-u}), $u_t(x)$ is a weighted average of $v_t(x)$ over a compact subset of the disk and therefore it must have the same growth rate as well. This proves
  (\ref{eq:beta-growth}).
\end{proof}

Before continuing further, we note that since the total mass of $u_t$ is increasing, the rate of growth of $u_t$ is the same as that of
$$
\hat u_t = \int_0^t u_s(x) ds,
$$
i.e.~
\begin{equation}
\beta_V = \limsup_{t \to \infty} \frac{1}{t} \log \int_{\mathbb{D}} {\hat u}_t(x) dA_{\hyp}(x). 
\end{equation}
In practice, we prefer to work with $\hat u_t$ since it is slightly easier  to estimate.

\subsection{Potentials supported on gardens}

We now turn our attention to Theorem \ref{sparse-analogue}.
For the proof, we may assume that $\supp u_0 \cap \mathcal G = \emptyset$.
In view of (\ref{eq:heat-derivative}), in order to obtain an upper bound for $\beta_V$, it suffices to prove a {\em non-concentration estimate} -- i.e.~to show that most of the mass of $u_t(x)$ is located outside of $\mathcal G$.

\begin{lemma}
\label{eq:main-fk-lemma}
If $p < p_0(R)$ is sufficiently small, then for any horoball $B \subset \mathcal G$ and any $t > 0$, the ratio
\begin{equation}
\mathcal Q_t(B, B^*) \, := \, \frac{\int_B {\hat u}_t(x) dA_{\hyp}(x)}{\int_{B^*} {\hat u}_t(x) dA_{\hyp}(x)}
\, \le \, Ce^{-R/2},
\end{equation}
where $B^* = \{ z \in \mathbb{D} : d_{\mathbb{D}}(z, B) \le R/2 \}$.
\end{lemma}

Temporarily assuming Lemma \ref{eq:main-fk-lemma}, note that since $V = 0$ on $\mathbb{D} \setminus \mathcal G$,
\begin{equation}
\int_{\mathbb{D}} V(x) {\hat u}_t(x) dA_{\hyp}(x) \le Ce^{-R/2} \int_{\mathbb{D}} {\hat u}_t(x) dA_{\hyp}(x), \qquad t > 0.
\end{equation}
Combining with (\ref{eq:heat-derivative}), we arrive at the inequality
\begin{equation}
\frac{d}{dt} \int_{\mathbb{D}} {\hat u}_t(x) dA_{\hyp}(x) \le C(u_0, V) + Ce^{-R/2} \int_{\mathbb{D}} {\hat u}_t(x) dA_{\hyp}(x),
\end{equation}
from which Theorem \ref{sparse-analogue} follows after integration.

\subsection{Proof of the non-concentration estimate}

For each $t \in (0, \infty)$, consider the function
$$
\hat v_t(x) := \int_0^t v_s(x) ds = \int_0^t \mathbb{E}_x \biggl \{ \exp \int_0^s V(B_r) dr \biggr\} ds.
$$
Inspecting (\ref{eq:total-mass-u}), we see that the total mass of $\hat u_t$ can be computed using Brownian paths that emanate from a point in the support of $u_0$.
To compute $\int_{B^*} {\hat u}_t(x) dA_{\hyp}(x)$, we  only need to consider Brownian paths that cross $\partial  B^*$ before time $t$, where the contribution of a Brownian path to the integral comes from the duration of time it spends in $B^*$. 
We let $\Pi$ denote the collection of all such paths. We partition $\Pi$ into disjoint collections $\Pi(x_0, t_0)$, indexed by $x_0 \in \partial B^*$ and $0 < t_0 < t$, where $x_0$ and $t_0$ are respectively the location and time of first entry into $B^*$.

For a pair $(x_0, t_0)$ as above, consider the ratio
\begin{equation}
\mathcal Q_t^{x_0,t_0}(B, B^*)
 \, := \, \frac{\int_B {\hat v}_{t - t_0}(x_0) dA_{\hyp}(x)}{\int_{B^*} {\hat v}_{t-t_0}(x_0) dA_{\hyp}(x)}.
\end{equation}
To prove Theorem \ref{sparse-analogue}, it suffices to show that
$\mathcal Q_t^{x_0,t_0}(B, B^*) \le Ce^{-R/2}$ for all admissible pairs $(x_0, t_0)$ with a uniform constant $C > 0$. Here, we are using the following elementary fact: if $(X, \mu)$ is a measure space and $u, v: X \to [0, \infty)$ are positive functions on $X$, then
$\int_X u(\xi) d\mu \, \bigl / \,\int_X v(\xi) d\mu \le \sup_{\xi \in X} u(\xi)/v(\xi)$.

Now if Brownian motion enters a horoball, it does not want to stay there for very long:
according to Lemma \ref{bm-hates-horoballs2}, the probability that the excursion time $\tau_{B^*}  > L$ is at most $Ce^{-\gamma L}$ for some constant $\gamma > 0$. Recall that this means that at time $t_0 + \tau_{B^*}$, Brownian motion will leave $B^*$ forever and never return again.
The Feynman-Kac formula provides a bonus for a path to stay in $B^*$, however, as we shall now explain, if the exponent $p > 0$ is small, this bonus is negligible. 

For this purpose, the coarse estimate
$$
\exp \biggl (\int_{t_0}^{t_0+\tau_{B^*}} V(B_s)ds \biggr ) \le e^{p(L+1)}
$$
is sufficient.
The above estimate tells us that in the range $p \in (0, \gamma - \delta)$, if we want to determine the denominator 
$\int_{B^*} {\hat v}_{t-t_0}(x_0) dA_{\hyp}(x)$ up to an additive error of $\varepsilon$, we may restrict our attention to Brownian paths for which $\tau_{B^*} < L_0(\varepsilon)$. 
By making $p > 0$ sufficiently small, we can ensure that $1 < e^{pL_0} < 1  + \varepsilon$ which means that the exponential term in the Feynman-Kac formula (\ref{eq:FK}) is essentially frozen during an excursion in $B^*$.
Clearly, similar remarks hold for the numerator as well.

The above considerations show that $ \mathcal Q_t^{x_0,t_0}(B, B^*)$ can be estimated from above by  the ratio of the expected times that Brownian motion spends in $B$ and $B^*$ when started at $x_0$ and simulated for time $t - t_0$\/:
\begin{equation}
\label{eq:estimating-quotient}
\mathcal Q_t^{x_0,t_0}(B, B^*) \lesssim \frac{\int_{B} g_{t-t_0}(x_0, x) dA_{\hyp}(x)}{\int_{B^*} g_{t-t_0}(x_0, x) dA_{\hyp}(x)}.
\end{equation}
In view of Lemma \ref{s-replacement} and Lemma \ref{bm-hates-horoballs}(ii), this is $\le Ce^{-R/2}$.
This completes the proof of Lemma \ref{eq:main-fk-lemma}.

\section{Brownian integral means spectrum}
\label{sec:BIMS}

We now return to the original problem involving conformal mappings. 
To prove Theorem \ref{sparse-thm}, we translate the Feynman-Kac argument from the previous section.
In this section, we give a direct translation which mimics the previous section as much as possible.
Later, we will give a slightly simplified account of this argument  which does not involve Brownian motion.

Let $f : \mathbb{D} \to \mathbb{C}$ be a conformal mapping.
Fix $p \in \mathbb{C} \setminus \{0\}$ and consider the functions 
\begin{equation}
u_t(x) = |f'(x)^p|\cdot p_t(x)
\end{equation}
 and
\begin{equation}
\hat u_t(x) = \int_0^t u_s(x) ds =  |f'(x)^p| \cdot g_t(x).
\end{equation}
Differentiating, we discover \begin{align}
\frac{d}{dt} \int_{\mathbb{D}} |f'(x)^p| \cdot p_t(x) dA_{\hyp}(x) & = \frac{1}{2} \int_{\mathbb{D}} |f'(x)^p| \cdot \Delta_{\hyp} [p_t(x)] dA_{\hyp}(x), \\
& =  \frac{1}{2} \int_{\mathbb{D}} \Delta_{\hyp} |f'(x)^p| \cdot p_t(x) dA_{\hyp}(x), \\
\label{eq:growth-of-u}
& =  \int_{\mathbb{D}} V(x) |f'(x)^p| \cdot p_t(x) dA_{\hyp}(x),
\end{align}
where $V = \frac{1}{2} \cdot |p|^2 |n_f/\rho|^2$ and $n_f := f''/f'$ is the {\em non-linearity} of $f$.

The above identity suggests that  estimating integral means is quite similar to studying the growth rate of  solutions of parabolic equations given by Feynman-Kac formula; however, now the mass flows differently: in the Feynman-Kac setting, the mass increases at rate $V$ near a point equally in all directions, while in the conformal setting, the mass is spread out unevenly.
 
The other property of the ``Brownian spectrum'' of a conformal mapping that we need is that it is larger than the usual integral means spectrum:
\begin{lemma}
\label{domination-lemma}
If we define 
\begin{equation}
\label{eq:bb1}
\tilde \beta_f(p)  := \limsup_{t \to \infty} \frac{1}{t} \log \int_{\mathbb{D}} u_t(x) dA_{\hyp}(x),
\end{equation}
then   $\beta_f(p) \le \tilde \beta_f(p)$ for any conformal map $f$. 
\end{lemma}

\begin{proof}
Fix an $\varepsilon > 0$. Since the integral means $\int_{S_r} |f'(z)^p| \, d\theta$ are increasing in $r$, 
(\ref{eq:linear-displacement}) shows that
\begin{equation}
\label{eq:ut-integral}
 \int_{\mathbb{D}} u_t(x) dA_{\hyp}(x) \ge c \int_{S_{r((1-\varepsilon)t)}} |f'(z)^p| \, d\theta, \qquad t \ge t_0(\varepsilon),
\end{equation}
where $r((1-\varepsilon)t)$ is chosen so that $d_\mathbb{D}(0, S_{r((1-\varepsilon)t)}) = (1-\varepsilon)t$. Hence, $(1-\varepsilon) \beta_f(p) \le \tilde \beta_f(p)$. Since $\varepsilon > 0$ was arbitrary, the proof is complete.
\end{proof}

\begin{remark}In reality, with exponentially small probability, Brownian particles can travel farther than expected, so the two characteristics need not be equal (and in fact, are never equal unless both are 0).
\end{remark}

\subsection{Sparse conformal mappings}

We now restrict to the case when $f = \tilde w^{k \mu^+}$ with $\mu$ supported on the garden $\mathcal G \subset \mathbb{D}$ satisfying the sparsity condition. Since finitely many horoballs have no effect on the dimension, we may assume that
$$\mathcal G \subset \{z \in \mathbb{D}: r_\delta < |z| < 1\},$$
for any choice of $r_\delta \in [1/2, 1)$. More precisely, we may write $\mathcal G = \mathcal G_1 \sqcup \mathcal G_2$ where $\mathcal G_1$ consists of horoballs which intersect $B(0,r_\delta)$ and $\mathcal G_2$ consists of horoballs that do not. Disjointness ensures that the collection $\mathcal G_1$ is finite. If $\mu_1 = \chi_{\mathcal G_1} \cdot \mu$ and $\mu_2 = \chi_{\mathcal G_2} \cdot \mu$, then by Sto\"ilow factorization $\tilde w^{k \mu^+} =  \tilde w^{k \mu_1^+} \circ  \tilde w^{\mu_2^+}$. The equality
$$
\Mdim {\tilde w}^{k \mu^+}(\mathbb{S}^1) = \Mdim {\tilde w}^{k \mu_2^+}(\mathbb{S}^1)
$$
now follows since ${\tilde w}^{k \mu_2^+}(\mathbb{S}^1)$ intersects the closure of ${\tilde w}^{k \mu_2^+}(\mathcal G_1)$ in a finite set and the fact that analytic mappings are differentiable.

We now restrict our attention to the case when $0 < k < 0.49$. According to an (improved version of) E.~Dyn'kin's estimate \cite{dynkin} for non-linearity, for any $\delta > 0$, we can find an $0< r_\delta < 1$ so that
\begin{equation}
\label{eq:vbound}
V(z) \, \le \, V_2^\delta(z) \, := \, 
\begin{cases}
               Ck^2|p|^2  e^{-(2-2k-\varepsilon)S} + \delta,   &  S = d_{\mathbb{D}}(z, \mathcal G) < R/2, \\
                              Ck^2|p|^2  e^{-(2-2k-\varepsilon)R/2} + \delta, & d_{\mathbb{D}}(z, \mathcal G) \ge R/2 \text{ and }|z| > r_\delta, \\
               18 k^2|p|^2 + \delta,               &  |z| < r_\delta,
           \end{cases}
\end{equation}
see  Appendix \ref{sec:dynkin}.
 The assumption $k < 0.49$ guarantees that $V_2 := V_2^\delta - \delta$ decays fast enough away from the support of $\mu$. It is possible that this is a merely technical restriction imposed by Dyn'kin's estimate and is not strictly necessary.
 
  The bound on the compact set $\{z : |z| < r_\delta\}$ is unimportant since it plays no role in determining the integral means spectrum. In order to define a function on the unit disk, we have used the Bloch norm bound  $\|\log f'\|_{\mathcal B} \le 6k$ guaranteed for a conformal mapping with a $k$-quasiconformal extension.

Like in the previous section, we prefer to work with $\hat u_t$ instead of $u_t$. Since the total mass 
of $u_t$ is increasing in $t$, we have
\begin{equation}
\label{eq:bb2}
\tilde \beta_f(p)  = \limsup_{t \to \infty} \frac{1}{t} \log \int_{\mathbb{D}} \hat u_t(x) dA_{\hyp}(x).
\end{equation}

Integrating (\ref{eq:growth-of-u}) with respect to time, we obtain
\begin{equation}
\label{eq:wbound}
\frac{d}{dt} \int_{\mathbb{D}} {\hat u}_t(x) dA_{\hyp}(x) \le C(p, f) + \int_{\mathbb{D}} V(x) {\hat u}_t(x) dA_{\hyp}.
\end{equation}
Since $V_2^\delta \le  C k^2|p|^2 e^{-R/2} + \delta$ on $\{x : r_\delta < |x| < 1 \} \setminus \bigcup B_j^*$, to prove Theorem \ref{sparse-thm2},  it suffices to show that
\begin{equation}
\label{eq:cm-quotient}
\frac{\int_{B^*} V_2^\delta(x){\hat u}_t(x) dA_{\hyp}(x)}{\int_{B^*} {\hat u}_t(x) dA_{\hyp}(x)} \le
 C k^2|p|^2 e^{-R/2} + \delta, \qquad k|p| < c/R,
\end{equation}
for each $B^*$ associated to a horoball $B$ in $\mathcal G$. Assuming (\ref{eq:cm-quotient}) for the moment, we have
\begin{equation}
\label{eq:cm-quotient2}
\frac{\int_{\mathbb{D}} V(x){\hat u}_t(x) dA_{\hyp}(x)}{\int_{\mathbb{D}} {\hat u}_t(x) dA_{\hyp}(x)} \le
 C k^2 |p|^2 e^{-R/2} + \delta
\end{equation}
which gives $\tilde \beta_{\tilde w^{k\mu^+}}(p) \le Ce^{-R/2}k^2|p|^2 + \delta$ after integration. 
Using Lemma \ref{domination-lemma} and the fact that $\delta > 0$ was arbitrary proves
$\beta_{\tilde w^{k\mu^+}}(p) \le Ce^{-R/2}k^2|p|^2$, which is the statement of Theorem \ref{sparse-thm2}.

\subsection{Non-concentration estimate}

Before proving (\ref{eq:cm-quotient}), let us first show the slightly simpler statement
\begin{equation}
\label{eq:cm-quotient3}
\frac{\int_{B^*} V_2^\delta(x){\hat u}_\infty(x) dA_{\hyp}(x)}{\int_{B^*} {\hat u}_\infty(x) dA_{\hyp}(x)} \le
 C k^2 |p|^2 e^{-R/2} + \delta.
\end{equation}
This relies on the ``freezing lemma'' which says that for the purpose of estimating integral means,
 $|f'(z)^p|$ is essentially constant on horoballs:

\begin{lemma}
\label{freezing-lemma}
Suppose $6k|p| < 0.49$. Then,
\begin{equation}
\label{eq:freezing1}
 \int_{B^*}  | f'(z)^p | \, g_\infty(z) dA_{\hyp}(z) \asymp  \diam B^* \cdot |f'(z_{B^*})|^p.
\end{equation}
If additionally  $k < 0.49$ and $k|p| < c/R$, then
 \begin{equation}
 \label{eq:freezing2}
 \int_{B^*} V_2(z) | f'(z)^p |  \, g_\infty(z) dA_{\hyp}(z) \asymp Ck^2 |p|^2 e^{-R/2} \cdot   \diam B^* \cdot |f'(z_{B^*})|^p.
 \end{equation}
 \end{lemma}

 \begin{proof}
 The lemma follows from the Bloch norm bound $\|\log f'\|_{\mathcal B} \le 6k$ mentioned earlier. 
For the first statement, it is clear that the integral over the top half of $B^*$ is comparable
to $ \diam B^* \cdot |f'(z_{B^*})|^p$. This proves the lower bound. 
For the upper bound, the Bloch norm bound implies
$$
 \int_{B^*}  | f'(z)^p | \, g_\infty(z) dA_{\hyp}(z) \lesssim
 | f'(z_{B^*})^p | \cdot \int_{B^*} \biggl (\frac{1-|z_{B^*}|}{1-|z|} \biggr)^{6k|p|} \,\frac{|dz|^2}{1-|z|}.
 $$
 The right hand side is integrable provided that $6k|p| < 1/2$. By asking for $6k|p| < 0.49$, we ensure that the integral over the top half of $B^*$ controls the integral over the bottom half.

 For $0 \le m \le n = \lceil R/2 \rceil$, set $B^m = \{z \in B^* : d_\mathbb{D}(z, B) \le m\}$, $S^m = B^m \setminus B^{m-1}$, $S^0 = B$.
 Note that  the condition $k|p| < c/R$ ensures that
 $ |f'(z_{B^m})|^p \asymp  |f'(z_{B^*})|^p$ for all $m=0,1,\dots, n$.
For the second statement, we subdivide $B^* \setminus B$ into shells
 and apply
 the estimate (\ref{eq:vbound}) on each shell. As $m$ increases, the contributions of $S^m$ decay exponentially (since $2 - 2k > 1.02$), so the integral in
  (\ref{eq:freezing2}) is dominated by the integral over $B$.
  \end{proof}
 
Equation (\ref{eq:cm-quotient3}) follows after dividing  (\ref{eq:freezing2}) by  (\ref{eq:freezing1}).
It remains to replace ``$\infty$'' with ``$t$.'' This last step is not necessary when one uses the Becker-Pommerenke method below, nevertheless it is quite easy with help of Lemma \ref{t-replacement}. Since $V_2$ is small on shells $n$ and $n-1$, we need not worry about them in the numerator.
Estimating the numerator (with two shells removed) from above and the denominator from below, we get
\allowdisplaybreaks
\begin{align*}
\int_{B^{n-2}} V_2(z) |f'(z)^p| \, g_t(z) dA_{\hyp}(z)  & \lesssim |f'(z_{B^*})^p| \cdot \sum_{m=0}^{n-2}  V_2(z_{B^m}) \cdot g_t(z_{B_m}), \\
  & \lesssim |f'(z_{B^*})^p|  \cdot g_t(z_{B_{n-2}}) \cdot \sum_{m=0}^{n-2}  \frac{V_2(z_{B^m})}{e^{n-m}}, \\
    & \lesssim |f'(z_{B^*})^p| \cdot g_t(z_{B_{n-2}}) \cdot \frac{k^2|p|^2}{e^{n}}
\end{align*}
and
$$
\int_{B^*} |f'(z)^p| \, g_t(z) dA_{\hyp}(z) \gtrsim |f'(z_{B^*})^p| \cdot g_t(z_{B_{n-2}}).
$$
Hence, (\ref{eq:cm-quotient}) follows after division. This completes the proof of 
 Theorem \ref{sparse-thm2}.

\section{Becker-Pommerenke argument}
\label{sec:bp-argument}

We now give a slightly simplified account of the above argument using the framework of Becker and Pommerenke for estimating integral means as presented in \cite{qcdim}.
For convenience, we switch over to the upper half-plane.
Let $\mathcal G = \bigcup B_j$ be a collection of horoballs in $\mathbb{H}$ such that $d_{\mathbb{H}}(B_i, B_j) > R$ for $i \ne j$,
and suppose $\mu \in M(\Hbar)$ is a Beltrami coefficient with $\|\mu\|_\infty \le 1$ whose support is contained in the reflected garden $\overline{\mathcal G}$.
Without loss of generality, one can assume that both $\mathcal G$ and $\mu$
are invariant under $z \to z+1$ so that $f = \tilde w^{k \mu}$ satisfies the periodicity condition $f(z+1)=f(z)+1$ needed in \cite{qcdim}. In this case, $\mathcal G$ is contained in $\{ z \in \mathbb{H} : \im z < 2 \}$. For a horoball $B$ in the upper half-plane, let $B^* = \{z \in \mathbb{H}: d_{\mathbb{H}}(z, B) < R/2 \}$. The separation condition ensures us that the horoballs $B_j^*$ are disjoint. As usual, we use the notation $z_{B^*}$ to denote the top point of $B^*$. 

Let $A(t)$ be the rectangle $[0,1] \times [t, 2] \subset \mathbb{H}$ and consider the function
 \begin{equation}
 \label{eq:u-def}
u(t) :=  \int_{A(t)}  | f'(z)^p |  \dzy.
\end{equation}
Taking the second derivative like in \cite[Section 7]{qcdim}, we get
\begin{equation}
\label{eq:improved-smirnov}
u''(t) = \frac{|p|^2}{4t^2}  \int_{A(t)}  | f'(z)^p | \,  \biggl | \frac{2n_f}{\rho_{\mathbb{H}}} \biggr |^2 \dzy + \mathcal O_p(1),
\end{equation}
where the $\mathcal O_p(1)$ term comes from the top side of $A(t)$. Since we are interested in the asymptotic behaviour as $t \to 0^+$, this term is harmless. For a set $K \subset \mathbb{H}$, let us write
$$
\mathcal Q(K) := \frac{ \int_{K}  | f'(z)^p | \,  \bigl | \frac{2n_f}{\rho_{\mathbb{H}}} \bigr |^2 \dzy}{ \int_{K}  | f'(z)^p | \dzy}.
 $$
 In this formalism, we must show:
 \begin{lemma} If $6k|p| < \min(0.49, c/R)$ then
  $$\mathcal Q(A(t)) \le Ce^{-R/2} k^2$$ for $t > 0$ sufficiently small.
 \end{lemma}
 
 \begin{proof}
The upper half-plane analogue of the freezing lemma (Lemma \ref{freezing-lemma}, with $|dz|^2/y$ replacing $g_\infty dA_{\hyp}$) shows that
$$
\mathcal Q(B^*_j \cap A(t)) \le Ce^{-R/2}k^2,
$$
for any horoball $B^*_j$.
Outside $\mathcal G^* = \bigcup B_j^*$, the  weight $ \bigl | \frac{2n_f}{\rho_{\mathbb{H}}} \bigr |^2$ is small, and so the quotient $\mathcal Q(A(t) \setminus \mathcal G^*)$ is small as well.
Putting these estimates together proves the lemma.
 \end{proof}
 
 With the above lemma, Theorem \ref{sparse-thm} follows from \cite[Lemma 7.1]{qcdim}.
 
 \begin{remark}
Define a ``boat'' of order $\beta$ to be an image of
 $$
 B(1,0,\beta) = \bigl \{ z \in \mathbb{H} \, : \, -1 < x < 1,  \, x^{\beta} < y < 1 \bigr \}
 $$ 
under an affine mapping $z \to az +b$ with $a>0$, $b\in \mathbb{R}$.
 The above argument also works in the case when $\mathcal G$ is a union of boats 
$\bigcup_{j=1}^\infty B(a_j, b_j, \beta_j)$ which are located at least a hyperbolic distance $R$ apart, provided the orders $\{\beta_j\}$ are bounded. The key point is that there exists an 
 $\varepsilon > 0$ sufficiently small so that the integral $\int_{B_j} |dz|^2/y^{1-\varepsilon}$ is bounded by a constant, where $B_j$ ranges over the boats in $\mathcal G$. We leave the details to the interested reader.
 \end{remark}
 
 \appendix
 
 \section{Dyn'kin's estimate}

\label{sec:dynkin}

In \cite[Theorem 1]{dynkin}, Dyn'kin proved a general estimate for the non-linearity of conformal mappings with quasiconformal extensions:
 
 \begin{lemma}
  \label{dynkin-lemma}
  Suppose $f: \mathbb{D} \to \mathbb{C}$ is a conformal mapping which has a quasiconformal extension to the plane with dilatation $\mu$ with $\| \mu \|_\infty \le k$ for some $0<k<1$. Then,
 \begin{equation}
 \label{eq:dynkin}
 \biggl | \frac{n_f}{\rho}(z) \biggr | \le C_k (1-|z|)^{1-k} \biggl [ 1  + \int_{1-|z|}^1 \frac{\omega(z,t)}{t^{2-k}} dt \biggr ], \qquad |z| < 1,
 \end{equation}
 where
 $$
 \omega(z,t) = \biggl ( \frac{1}{\pi t^2} \int_{|\zeta-z| \le t} |\mu(\zeta)|^2 |d\zeta|^2 \biggr )^{1/2}.
 $$
 Here, the constant $C_k$ can be taken to be non-decreasing in $k \in (0,1)$.
 \end{lemma}

If one is  interested in utilizing only the support of $\mu$, Dyn'kin's technique yields a slightly better estimate: 
 
 \begin{lemma} 
   \label{dynkin-lemma2}
Suppose $f: \mathbb{D} \to \mathbb{C}$ is a conformal mapping which has a quasiconformal extension to the plane with dilatation $\mu$ with $\| \mu \|_\infty \le k$ for some $0<k<1$. Then,
 \begin{equation}
 \label{eq:dynkin2}
 \biggl | \frac{n_f}{\rho}(z) \biggr | \le C'_k (1-|z|)^{1-k} \biggl [ 1  + \int_{1-|z|}^1 \frac{\tilde \omega(z,t)}{t^{2-k}} dt \biggr ], \qquad |z| < 1,
 \end{equation}
 where
 $$
\tilde \omega(z,t) = k \cdot \biggl ( \frac{|\supp \mu \cap B(z, t)|}{|B(z, t)|} \biggr )^{\frac{1}{1+k}}.
 $$
 Again, the constant $C'_k$ can be taken to be non-decreasing in $k \in (0,1)$.
 \end{lemma}
 
 The proof of the above lemma is nearly identical to that of Dyn'kin's theorem, so we only give a sketch of the argument and explain where the improvement comes from. Set $r = 1-|z|$,
 $B_j = B(z, 2^j r)$ and $E_j = \supp \mu \cap B(z, 2^j r)$. By the Cauchy-Green formula and the elementary bound $|f(z)| \le C_k$ for $|z| \le 2$, we have
 $$
 |f''(z)| \le \biggl | \frac{2}{\pi} \int_{1 < |\zeta| < 2} \frac{\partial f}{\partial \overline{\zeta}} 
 \frac{|d\zeta|^2}{(\zeta - z)^3}\biggr | + C_2, \qquad z \in \mathbb{D}.
 $$
 Applying the Cauchy-Schwarz inequality, we see that the contribution of the annulus $\{\zeta : 2^j r < |\zeta - z| < 2^{k+1}r\}$ does not exceed
 $$
 \frac{1}{(2^j r)^3} \biggl ( \int_{E_j} |\mu(\zeta)|^2 |d\zeta|^2 \biggr )^{1/2}  \Biggl ( \int_{E_j} \biggl |\frac{\partial f}{\partial \zeta} \biggr |^2 |d\zeta|^2 \Biggr )^{1/2}.
 $$
Dyn'kin's estimates the first term using the characteristic $\omega$; to estimate the second term, he replaces the integrand with the Jacobian of $f$ and uses the coarse bound
 $$
|f(E_j)| \le |f(B_j)|.
 $$
In our setting,  Astala's area distortion theorem \cite[Theorem 13.1.5]{AIM} yields the stronger estimate
  $$
|f(E_j)| \le C_k |f(B_j)| \biggl (  \frac{|E_j|}{|B_j|} \biggr )^{\frac{1-k}{1+k}}.
 $$
 The use of this stronger estimate explains why the exponent in  Lemma \ref{dynkin-lemma2} is $\frac{1}{2} + \frac{1}{2} \cdot \frac{1-k}{1+k} = \frac{1}{1+k}$ compared to the exponent in  Lemma \ref{dynkin-lemma} which is only $\frac{1}{2}$.

In this paper, we utilize the above estimate in a slightly different form.
Let $(\supp \mu)^+ \subset \mathbb{D}$ be the reflection of the support of $\mu$ in the unit circle. 
 In terms of the hyperbolic distance from $z$ to $(\supp \mu)^+$, the estimate says:
 
 \begin{corollary}
Suppose $0 < k < 1$, $d_{\mathbb{D}} \bigl (z, (\mathcal \supp\, \mu)^+ \bigr ) > L$ and $d_\mathbb{D}(z, 0) > L$. Then,
$$
\biggl |\frac{n_f}{\rho}(z) \biggr | \lesssim k \cdot Le^{-(1-k)L} + (1-|z|)^{1-k},
$$ 
where the implicit constant can be taken to be non-decreasing in $k \in (0,1)$.
 \end{corollary}

\begin{proof}
For $j = 0, 1, 2, \dots, \lfloor L \rfloor$, the assumptions on the support of $\mu$ imply that the Euclidean area 
$$
\bigl |  \supp \mu \cap B \bigl (z, e^j(1-|z|) \bigr ) \bigr | \le C(1-|z|)^2e^{3j-L},$$
 and therefore 
$\omega \bigl(z, e^j(1-|z|) \bigr ) \lesssim k \cdot e^{(j-L)/2}$ and 
$\tilde \omega \bigl(z, e^j(1-|z|) \bigr ) \lesssim k \cdot e^{(j-L)/(1+k)}$.
Hence, the right hand side of (\ref{eq:dynkin2}) is bounded above by
 \begin{equation}
 \label{eq:append1}
 C(1-|z|)^{1-k} \biggl (1 + \int_{e^L(1-|z|)}^\infty \frac{k}{t^{2-k}} dt
 + \sum_{j=0}^{\lfloor L \rfloor} \int_{e^j(1-|z|)}^{e^{j+1}(1-|z|)} \frac{k \cdot Ce^{(j-L)/(1+k)}}{t^{2-k}} dt \biggr ),
 \end{equation}
After opening the brackets, the second term in  (\ref{eq:append1}) is
\begin{equation}
\label{eq:second-term}
 = \frac{Ck}{2-k} (1-|z|)^{1-k} \bigl [e^L(1-|z|) \bigr ]^{-(1-k)} 
   \asymp Ck \cdot e^{-L(1-k)},
\end{equation}
while the $j$-th term in the sum is comparable to
 $$
C(1-|z|)^{1-k}   \int_{e^j(1-|z|)}^{e^{j+1}(1-|z|)} \frac{k \cdot e^{(j-L)/(1+k)}}{t^{2-k}} dt \asymp  Ck \cdot e^{-j(1-k)}
  \cdot e^{(j-L)/(1+k)}.
 $$
Since $1/(1+k) - (1-k) > 0$ for any $0 < k < 1$, this is an increasing geometric series in $j$. In particular, each term in the sum is bounded by the last term with $j=L$. 
Putting these estimates together gives the corollary. 
\end{proof}

 \section{Wiggly potentials}
In this appendix, we give another entry in the dictionary between  integral means spectra of conformal maps and perturbations of the Laplacian.
 Call a potential $V(x)$ {\em wiggly} if there exists $R, \alpha > 0$ such that
$\int_B V(x) > \alpha$
over any ball of hyperbolic radius $R$.
\begin{theorem}
For a wiggly potential $V$, there exists a constant $c = c(\alpha, R) > 0$ such that $\beta_{p V} > c |p|^2$ for $0 \le p \le 1$.
\end{theorem}

This is reminiscent of a variant of a theorem of P.~Jones which we state for the integral means spectrum and use non-linearity instead of the Schwarzian derivative:

\begin{theorem}
\label{jones-thm}
Suppose $f: \mathbb{D} \to \mathbb{C}$ is a conformal mapping onto a domain with quasicircle boundary, such that  any
point $z \in \mathbb{D}$ is located within hyperbolic distance $R$ of a point $w \in \mathbb{D}$ with 
$$
\biggl |\frac{n_f}{\rho}(w_i) \biggr | > \alpha.
$$ Then, there exists a constant $c(\alpha, R) > 0$ so that
$$
\beta_f(p) > c(\alpha, R) |p|^2, \qquad |p| \le 1.
$$
\end{theorem}
The proofs of these theorems are quite simple and we leave them as exercises for the reader. Theorem \ref{jones-thm} can be viewed as a special case of 
\cite[Theorem 1.6]{qcdim}.

\bibliographystyle{amsplain}

\end{document}